\documentclass[12pt]{amsart}
\usepackage{amsfonts}
\usepackage{amsmath, amsthm, amssymb,color}
\usepackage{latexsym}
\usepackage{txfonts}
\usepackage{bm}
\usepackage{graphicx}
\usepackage{graphics}
\usepackage{epsfig}
\usepackage[polish,english]{babel}

\numberwithin{equation}{section}
\allowdisplaybreaks

\newtheorem{lemma}{Lemma}[section]

\newtheorem{theorem}[lemma]{Theorem}
\newtheorem{proposition}[lemma]{Proposition}
\newtheorem{definition}[lemma]{Definition}
\newtheorem{corollary}[lemma]{Corollary}
\newtheorem{example}[lemma]{Example}
\newtheorem{exercise}[lemma]{Exercise}
\newtheorem{remark}[lemma]{Remark}
\newtheorem{fig}[lemma]{Figure}
\newtheorem{tab}[lemma]{Table}

\newcommand{\bth}{\begin{theorem}}
\newcommand{\ethe}{\end{theorem}}
\newcommand{\bre}{\begin{remark}\em }
\newcommand{\ere}{\end{remark}}
\newcommand{\ble}{\begin{lemma}}
\newcommand{\ele}{\end{lemma}}
\newcommand{\bde}{\begin{definition}}
\newcommand{\ede}{\end{definition}}
\newcommand{\bco}{\begin{corollary}}
\newcommand{\eco}{\end{corollary}}
\newcommand{\bpr}{\begin{proposition}}
\newcommand{\epr}{\end{proposition}}

\newcommand{\bexer}{\begin{exercise}}
\newcommand{\eexer}{\end{exercise}}

\newcommand{\bexam}{\begin{example}\rm  }
\newcommand{\eexam}{ \end{example}}

\newcommand{\bfi}{\begin{fig}}
\newcommand{\efi}{\end{fig}}

\newcommand{\btab}{\begin{tab}}
\newcommand{\etab}{\end{tab}}

\def\E{{\mathbb{E}}}
\def\P{{\mathbb{P}}}
\def\R{{\mathbb{R}}}

\def\N{{\mathbb{N}}}

\def\Z{{\mathbb{Z}}}

\def\s{\sigma}

\def\B_e{B_{\eta}(e)}

\renewcommand{\a}{\alpha }
\renewcommand{\b}{\beta}

\renewcommand{\d}{\delta}

\newcommand{\1}{{\bf 1}}

\newcommand{\8}{\infty}

\newcommand{\ov}{\overline}

\newcommand{\wt}{\widetilde}

\newcommand{\blue}{\color{blue} }
\definecolor{darkblue}{rgb}{0,0,1}
\definecolor{darkgreen}{rgb}{0,1,0}
\definecolor{darkred}{rgb}{1, 0,0}

\newcommand{\garch}{{\rm GARCH}$(1,1)$}

\newcommand{\sta}{St\u aric\u a}

\newcommand{\bfR}{{\bf R}}

\newcommand{\rv}{random variable}

\newcommand{\sign}{{\rm sign}}

\newcommand{\bfTh}{\mbox{\boldmath$\Theta$}}

\newcommand{\bfPi}{\mbox{\boldmath$\Pi$}}

\newcommand{\beao}{\begin{eqnarray*}}
\newcommand{\eeao}{\end{eqnarray*}\noindent}

\newcommand{\beam}{\begin{eqnarray}}
\newcommand{\eeam}{\end{eqnarray}\noindent}

\newcommand{\beqq}{\begin{equation}}
\newcommand{\eeqq}{\end{equation}\noindent}

\newcommand{\bce}{\begin{center}}
\newcommand{\ece}{\end{center}}

\newcommand{\barr}{\begin{array}}
\newcommand{\earr}{\end{array}}

\newcommand{\stw}{\stackrel{w}{\rightarrow}}

\newcommand{\vague}{\stackrel{\lower0.2ex\hbox{$\scriptscriptstyle
                    \it{v} $}}{\rightarrow}}
\newcommand{\weak}{\stackrel{\lower0.2ex\hbox{$\scriptscriptstyle
                    \it{w} $}}{\rightarrow}}
\newcommand{\what}{\stackrel{\lower0.2ex\hbox{$\scriptscriptstyle
                    \it{\hat{w}} $}}{\rightarrow}}

\newcommand{\bdis}{\begin{displaymath}}
\newcommand{\edis}{\end{displaymath}\noindent}

\newcommand{\xto}{x\to\infty}

\newcommand{\regvary}{regularly varying}

\newcommand{\bbz}{{\mathbb Z}}

\newcommand{\bbs}{{\mathbb S}}

\newcommand{\con}{convergence}

\newcommand{\ds}{distribution}

\newcommand{\seq}{sequence}

\newcommand{\bfx}{{\bf x}}
\newcommand{\bfX}{{\bf X}}
\newcommand{\bfB}{{\bf B}}
\newcommand{\bfN}{{\bf N}}
\newcommand{\bfY}{{\bf Y}}
\newcommand{\bfy}{{\bf y}}
\newcommand{\bfA}{{\bf A}}

\newcommand{\bfO}{{\bf 0}}
\newcommand{\bfZ}{{\bf Z}}
\newcommand{\bfW}{{\bf W}}

\newcommand{\bfS}{{\bf S}}

\newcommand{\bfI}{{\bf I}}

\newcommand{\bali}{\begin{align}}
\newcommand{\eali}{\end{align}}

\renewcommand\d{{\mathrm d}}

\begin{document}

\bibliographystyle{alpha}
\title[Solutions to bivariate stochastic recurrence equations with
different tail indices]
      {Componentwise different
      tail solutions for bivariate stochastic recurrence equations \\-- with application to \garch\ processes --}
\today
\author[E. Damek]{Ewa Damek${}^\ast$}
\author[M. Matsui]{Muneya Matsui${}^\dagger$}
\author[W. Swiatkowski]{\foreignlanguage{polish}{Witold "Swi"atkowski}${}^\ast$ \vspace{2mm}\\
 {\scriptsize  University of \foreignlanguage{polish}{Wroclaw}${}^\ast$ and Nanzan University${}^\dagger$} }
\address{Department of Business Administration, Nanzan University,
18 Yamazato-cho Showa-ku Nagoya, 466-8673, Japan}
\email{mmuneya@nanzan-u.ac.jp}
\address{
Institute of Mathematics
University of \foreignlanguage{polish}{Wroclaw} 
Pl.  Grunwaldzki 2/450-384, Wroclaw, Poland}  
\email{edamek@math.uni.wroc.pl}
\email{witt.r@wp.pl}
 
\begin{abstract} 
 We study bivariate stochastic recurrence equations (SREs) motivated by
 applications to \garch\ processes. If coefficient matrices of SREs
 have strictly positive entries, then the Kesten result applies and it gives 
 solutions with regularly varying tails. Moreover, the tail indices are
 the same for all coordinates. However, for applications, this framework
 is too restrictive. We study SREs when coefficients are triangular matrices and prove that 
 the 
 coordinates of the solution may exhibit regularly varying tails   
 with different indices. We also specify each tail
 index together with its constant. The results are used to
 characterize regular variations of bivariate stationary \garch\ 
 processes. 
\vspace{2mm}  
\end{abstract}
\thanks{Ewa Damek's research was partly supported by the NCN grant
UMO-2014/15/B/ST1/00060. \\
\quad Muneya Matsui's research is partly supported by the JSPS Grant-in-Aid
for Young Scientists B (16k16023).
}
\maketitle 

\section{Introduction}
We consider the stochastic recurrence equation (SRE)
\begin{equation}\label{affine} 
\bfW_t = \bfA_t\bfW_{t-1}+\bfB_t, \quad t\in \N,
\end{equation}
where $(\bfA_t,\bfB_t)$ is an i.i.d.\ sequence, $\bfA_t$ are $d\times d$
 matrices, $\bfB_t$ are vectors and $\bfW_0$ is an initial distribution
 independent of the sequence $(\bfA_t, \bfB_t)$. Iterations
 \eqref{affine} generate a Markov chain $(\bfW_t)_{t\geq0}$ that is not necessarily stationary.
 Under mild contractivity hypotheses (see e.g. \cite{bougerol:picard:1992a, brandt:1986})
the sequence $\bfW_t$ converges in law to a random
variable $\bfW$ that is the unique solution of the equation
\begin{equation}
\label{dif recurrence} \bfW \stackrel{d}{=} \bfA \bfW+\bfB, \end{equation} where $\bfW$ is independent
of $(\bfA,\bfB)$ 
and the equation is meant in law. Here $(\bfA,\bfB)$ is a generic element of the sequence $(\bfA_t,\bfB_t)$.
If we put $\bfW_0=\bfW$ then the chain $\bfW_t$ becomes
stationary. Moreover, extending the set of indices to $\Z$
and taking an i.i.d.\ sequence $(\bfA_t, \bfB_t)_{t\in \Z}$ we can have a strictly stationary causal solution $\bfW_t$ to the equation
\begin{equation}\label{causal}
\bfW_t = \bfA_t\bfW_{t-1}+\bfB_t, \quad t\in \Z . 
\end{equation}
It is given by
\begin{equation*}
\bfW _t=\sum _{i=-\8}^t\bfA_t\cdots \bfA_{i}\bfB_{i-1}+\bfB_i\stackrel{d}{=}\bfW.
\end{equation*} 
There is considerable interest in studying various aspects of the
iteration \eqref{affine} and, in particular, the tail behaviour of
$\bfW$. The story started with Kesten \cite{kesten:1973} who
obtained fundamental results about tails of $\bfW_t$ in the case of
matrices $\bfA_t$ having non-negative entries.

Given $y=(y_1,\dots y_d)$ in the unit sphere $\mathbb{S}^{d-1}$, let
\begin{equation*}
y'\bfW=\sum _{j=1}^dy_jW_j,\quad \bfW =(W_1,\dots W_d).
\end{equation*}
Under appropriate assumptions Kesten \cite{kesten:1973}
proved that there is $\a >0$ and a function $e_{\a }$ on $\mathbb{S}^{d-1}$ such that
\begin{equation}\label{rege}
\lim _{x\to \8 }x^{\a }\P (y'\bfW>x)=e_{\a }(y),\quad y\in \mathbb{S}^{d-1}
\end{equation}
and $e_{\a }(y) > 0$ for $ y\in \mathbb{S}^{d-1}\cap [0,\8
)^d$. Later on an analogous result was proved by
Alsmeyer and Mentemeier \cite{alsmeyer:mentmeier2012} for invertible
matrices $\bfA$ with some irreducibility and density conditions.

 The density assumption was removed by Guivarc'h and Le
Page \cite{guivarch:lepage:2015} who developed the most general approach
to \eqref{affine} with signed $\bfA$ having possibly a singular
law. Moreover, their conclusion was stronger i.e.\ they obtained
existence of a measure $\mu $ on $\R ^d $ being the week limit of 
\begin{equation}
\label{regular}
 x^{\a}\P (x^{-1}\bfW\in \cdot ) \quad \mbox{when}\ x\to \infty .
\end{equation}
The latter means regular variation of $\bfW$.  
\footnote{If $\a \notin \N$
then \eqref{rege} implies regular variation of $\bfW$. 
If $\alpha\in \N$, the same holds with some additional conditions 
(see
\cite[Appendix C]{buraczewski:damek:mikosch:2016}). For more on regular
variation, we refer to Bingham et al. \cite{bingham:goldie:teugels:1987}
and Resnick \cite{resnick:1987,resnick:2007} in the univariate and
multivariate cases, respectively. }
and it was also proved also for \eqref{affine} with $\bfA$ being similarities
\cite{buraczewski:damek:guivarch2009}\footnote{$A$ is a similarity if
for every $x\in \bfR ^d$, $|Ax|=\| A\|\ |x|$.} i.e.\ when neither assumptions of
\cite{guivarch:lepage:2015} nor \cite{alsmeyer:mentmeier2012} are
satisfied. See \cite{buraczewski:damek:mikosch:2016} for an elementary
explanation of Kesten's result and other results mentioned above.

For all the matrices considered above we have the same
tail behavior in all directions, one of the reasons being a certain
irreducibility or homogeneity of the action of the group generated by
the support of the law of $\bfA$. But it does not always have to be like that. We may
imagine $\bfA=diag (A_{11},\dots A_{dd})$ being diagonal such that $\E
A_{ii}^{\a _i}=1$ and $\a _1, \dots \a _d$ are different (see
e.g.\ \cite{buraczewski:damek:2010},
\cite{buraczewski:damek:guivarch2009} and 
\cite[Appendix D]{buraczewski:damek:mikosch:2016}
). Then $W_1,\dots W_d$
are regularly varying with different exponents $\a _1,\dots \a _d$. In
such a case, if we want to say that $\bfW$ is regularly varying we need to
modify the notion. 
For more detailed explanation we refer to 
\cite[Chapter 4]{buraczewski:damek:mikosch:2016} as
well as the book by Resnick \cite[p.
203]{resnick:2007}, where non-standard regular variation appears in various contexts. 

Triangular matrices $\bfA$ do not fit into any of
the frameworks mentioned above and therefore considering them is a natural next step. However the existing methods cannot be applied and a new approach is needed. 
This is what we do here. We study $2\times 2$ upper triangular matrices
$\bfA=[A_{ij}]$ with positive entries 
(i.e.\ $A_{21}$ is the only one being zero) such that $\E  A_{ii}^{\a _i}=1$ and $\a _1\neq \a _2$. 
We prove that
\begin{itemize}
\item if $\a _1 >\a _2$ then $\bfW=(W_{1}, W_{2})$ is regularly varying with index $\a _2$.
\item if $\a _2 >\a _1$ then $W_{1}$ and $W_{2}$ are regularly varying with indices $\a _1$ and $\a _2$ respectively.
\end{itemize}
This is the content of Theorem \ref{thm:mainresult}. Then we study regular variation of $\bfW_t$ as a time series. In the first case we describe the spectral process $\bfY _t$ in the sense of \cite{basrak:segers:2009} corresponding to $\bfW _t$. It is of the form
\begin{equation}\label{spectral}
\bfY _t=\bfA _t\cdots \bfA _1\bfY _0,\end{equation}
where $\bfY _0 =\| \bfY_0\| {\mathbf \Theta}_0$, $\P (\| \bfY_0\|
>u)=u^{-\a },\,u\ge1$ and the law of $\mathbf{\Theta}_0$ is the spectral measure
of $\bfW $, see Proposition \ref{prop:fidm2}. 
In the second case we consider $W_{1,t}$ and $W_{2,t}$ separately (Lemma  \ref{lem:differenttail}).

Our results are interesting from the point of view of financial analysis and they apply to the squared volatility sequence $\bfW _t=(\s ^2_{1,t}, \s
^2_{2,t})$ of the bivariate GARCH(1,1) financial model, see Section
\ref{aplication}. Then $\bfW_t$ satisfies \eqref{affine} with matrices
$\bfA_t$ having non-negative entries. If all the entries of $\bfA_t$ are
strictly positive then the theorem of Kesten applies and both $\s
^2_{1,t}$ and $\s ^2_{2,t}$ are regularly varying with the same
index, see \cite{matsui:mikosch:2016}, \cite{mikosch:starica:2000}. But if this is not the case then we have to go beyond Kesten's
result and Theorem \ref{thm:mainresult} below enters into the picture. From the point of view of applications it is reasonable to relax
the assumptions on $\bfA_t$ because it allows us to capture a larger class of financial models.

When matrices $\bfA_t$ are upper triangular we may apply the above results to
obtain regular variation of $\s ^2_{1,t}$ and $\s ^2_{2,t}$. Namely if
$\a _1< \a _2$ then $\s ^2_{1,t}$ and $\s ^2_{2,t}$ are regularly varying
with indices $\a _1$ and $\a _2$ respectively. If $\a _1> \a _2$
then $(\s ^2_{1,t}, \s ^2_{2,t})$ is regularly varying with the index $\a
_2$ and we have a nice description of its spectral
process. Finally, in Propositions \ref{prop:spectral1} and \ref{spectralgarch} we study regular variation of the bivariate
GARCH(1,1) itself $\bfX _t=(\s_{1,t}Z_{1,t}, \s_{2,t}Z_{2,t})$.

 It turns out that the appearance of triangular matrices in
 \eqref{affine} generates a lot of technical complications, it is
 challenging and it is far from being solved in arbitrary dimension.  
Even for $2\times 2$ matrices, the case when $\a _1= \a _2$ is, in our
opinion, out of reach in full generality at the moment. There is a preprint
\cite{damek:zienkiewicz:2017} on that case with an extra
assumption that $A_{11}=A_{22}$. 

\section{Bivariate stochastic recurrence equations}

We start with the description of the model as well as 
conditions for stationarity of the related time series.
\subsection{The model}
We consider the bivariate SRE;
\begin{align}\label{bivSRE}
 \bfW_t =\bfA_t \bfW_{t-1}+\bfB_t, \quad t \in \Z, 
\end{align}
where 
\beam
\bfW_t= \left(
\barr{l} W_{1,t}  \\ W_{2,t}   \earr\right),\quad
\bfA_t=\left(\barr{cc} A_{1,t} & A_{2,t} \\
0 & A_{4,t}
\earr\right)\quad \mathrm{and}\quad \bfB_t= \left(
\barr{l} B_{1,t}  \\ B_{2,t}   \earr\right), 
\eeam
and an i.i.d. matrix sequence $(\bfA_t)$ and an i.i.d. vector sequence
$(\bfB_t)$. 
We assume $A_{i,t}>0$ a.s.\ $i=1,2,4$ and
$B_{i,t}>0$ a.s.\ $i=1,2.$ For our purpose, it is convenient to
write the SRE in a coordinate-wise form; 
\begin{align}
 W_{1,t} & = 
 A_{1,t}W_{1,t-1} + D_t, \label{uniSRE1} \\
 W_{2,t} & = A_{4,t} W_{2,t-1} + B_{2,t}, \label{uniSRE2} 
\end{align}
where $D_t:=B_{1,t}+A_{2,t}W_{2,t-1}.$
We sometimes omit the subscript $0$ in $A_{i,0}$, $B_{i,0}$ and
$W_{i,0}$, 
etc.\ and just write
$A_{i}$, $B_{i}$ and $W_{i}$ if they are stationary. 
For further convenience we denote
for $t\in\mathbb{Z}$
\begin{align*}
  \bfPi_{t,s} & =\bfA_t\cdots \bfA_s,\,t\ge s,\quad
 \bfPi_{t,s}=\bfI,\,t<s\quad \mathrm{and}\quad \bfPi_t =\bfPi_{t,1}, \\
  \Pi_{t,s}^{(i)}& =\prod_{j=s}^t A_{i,j},\,t\ge s,\,i=1,2,4\quad
 \mathrm{and}\quad \Pi_{t,s}^{(i)}=1,\,t<s\quad \mathrm{and}\quad \Pi_t^{(i)}=\Pi_{t,1}^{(i)},
\end{align*}
where $\bfI$ is the bivariate identity matrix. 
For a vector $\bfx\in \R^d$, $|\bfx|$ denotes its Euclidean norm and
for a $d\times d$ matrix $\bfA$ we use the matrix norm;
\[
 ||\bfA||=\sup_{\bfx\in\R^d,\,|\bfx|=1}|\bfA \bfx|. 
\]
\subsection{Stationarity}
Starting from Kesten \cite{kesten:1973} there is a series of results \cite{brandt:1986,bougerol:picard:1992a}
for the
existence of stationary solution for SRE (see also \cite{buraczewski:damek:mikosch:2016}). 
The notion of the ``so called '' {\em top Lyapunov exponent} 
\beao
\gamma=\inf_{n\ge 1} n^{-1} \E \log \| 
\bfPi_{n}\| 
\eeao
associated with the \seq\ $(\bfA_t)$ is always essential.
If $\gamma$ is negative and some logarithmic moment conditions are satisfied, then SRE
 \eqref{bivSRE} has a unique strictly stationary solution
 (\cite{bougerol:picard:1992a} or 
 \cite[Theorem 4.1.4]{buraczewski:damek:mikosch:2016}). 

In our setting, $\gamma <0$ if there is $\varepsilon>0$ such that
\beam\label{epsilon_moment}
\E A_1^\varepsilon<1,\quad \E A_4^\varepsilon<1\quad \mbox{and}\quad\E
A_2^\varepsilon<\infty. 
\eeam
We are going to show this.
Without loss of generality we may assume that $\varepsilon<1$. 

First observe that by the Jensen's inequality 
\beam \label{lyapunov_jensen}
\gamma=\inf_{n\ge
1}(n\varepsilon)^{-1}\E\log\|\bfPi_n\|^\varepsilon\le\inf_{n\ge
1}(n\varepsilon)^{-1}\log\E\|\bfPi_n\|^\varepsilon. 
\eeam
Secondly, we decompose the matrix $\bfA_t=\bfS_t+\bfN_t$ into the sum of a diagonal and a nilpotent one, where
\beao
\bfS_t=\left(\barr{cc} A_{1,t} & 0 \\
0 & A_{4,t}\earr\right)\quad \mathrm{and}\quad 
\bfN_t=\left(\barr{cc} 0 & A_{2,t} \\
0 & 0\earr\right), 
\eeao
so that $\bfS_i\bfS_j$ is diagonal and $\bfS_i\bfN_j$,
$\bfN_i\bfS_j$ are nilpotent. 
Then we write 
\[
 \bfPi_n=\bfA_n\ldots\bfA_1=(\bfS_n+\bfN_n)\cdots (\bfS_1+\bfN_1)
\]
as the sum of $2^n$ products. Moreover, observe that the product of bivariate matrices vanishes if the matrices
have only zero entries except in the top right corner, and therefore 
only terms including at most one $\bfN_i$ are nonzero. Hence we have 
\beao
\|\bfPi_n\|^\varepsilon=\|\sum_{i=0}^n
\bfS_n\ldots\bfS_{i+1}\bfN_i\bfS_{i-1}\ldots\bfS_1\|^\varepsilon\le
\sum_{i=0}^n\|\bfS_n\ldots\bfS_{i+1}\bfN_i\bfS_{i-1}\ldots\bfS_1\|^\varepsilon, 
\eeao
where $\bfN_0=\bfS_0=\bfS_{-1}=\bfI$. 
Notice that
$\bfS_n\ldots\bfS_{i+1}\bfN_i\bfS_{i-1}\ldots\bfS_1,\,i \neq 0$ is a nilpotent
matrix and its only nonzero entry is
$\Pi_{n,i+1}^{(4)}A_{2,i}\Pi_{i-1,1}^{(1)}$. Moreover, all terms in the product are independent, so we may write
\beao
\E\|\bfPi_n\|^\varepsilon\le
\E\sum_{i=1}^n\left(\Pi_{n,i+1}^{(4)}A_{2,i} \Pi_{i-1,1}^{(1)}\right)^\varepsilon
=\sum_{i=1}^n\left(\E A_4^\varepsilon\right)^{n-i}\E A_2^\varepsilon
\left(\E A_1^\varepsilon\right)^{i-1} \le nc\rho^n, 
\eeao
where $\rho=\max(\E A_1^\varepsilon, \E A_4^\varepsilon)<1$ and $c=\rho^{-1}\E A_2^\varepsilon<\infty$.

Eventually, we can estimate the top Lyapunov exponent, 
\beao
\gamma\le\inf_{n\ge 1}(n\varepsilon)^{-1}\log nc\rho^n=\varepsilon^{-1}\inf_{n\ge 1}\left(\frac{\log nc}{n}+\log \rho\right)=\varepsilon^{-1}\log \rho<0
\eeao
which is what we needed. If we assume additionally that
\begin{equation}\label{logB}
\E \log ^+ |B| <\infty,  
\end{equation}
then we may conclude that there exists an a.s.\ unique causal strictly
stationary ergodic 
solution to SRE \eqref{bivSRE} given by the infinite
 series, 
 \begin{align}
\label{solbibSRE}
  \bfW_t=\sum_{i=-\infty}^t \bfPi_{t,i+1}\bfB_i. 
 \end{align}
Moreover \eqref{uniSRE1} and \eqref{uniSRE2} are in agreement with  \eqref{solbibSRE}.

Due to Theorem 1 of \cite{brandt:1986} (see also \cite[Section 2.1]{buraczewski:damek:mikosch:2016}), 
 a strictly
 stationary positive solution for \eqref{uniSRE2} exists if 
 \[
   \E \log A_{4} <0\quad \text{and}\quad \E \log^+ B_{2}<\infty.  
 \]
 Notice that from stationarity condition of bivariate case, those for
 component wise SREs are automatically satisfied, i.e.\ $\E\log A_{1}$ and
 $\E\log A_{4}$ are smaller than $0$, which ensures strictly stationary solution
 \[
  W_{2,t} = \sum_{i=1}^\infty 
  \Pi_{t,t+2-i}^{(4)} B_{2,t+1-i}. 
 \]
  Now consider this stationary version $(W_{2,k})$ for $(D_t)$ of SRE
  \eqref{uniSRE1}. Since 
  $W_{2,k}$ is independent of elements
  $(A_{2,\ell},B_{1,\ell}),\,\ell>k$ of i.i.d.\
  sequence $(A_{2,t},B_{1,t})$, we observe that 
 \begin{align}
 \label{eq:processD}
  D_t = B_{1,t} + A_{2,t}W_{2,t-1}
 \end{align}
 is stationary and ergodic, so is the sequence $(A_{1,t},D_t)$. 
 Then from Theorem 1 of \cite{brandt:1986} the series $(W_{1,t})$ has the stationary solution given by 
 \begin{align}
 \label{infseriesw1}
  W_{1,t} = \sum_{i=0}^\infty \Pi_{t,t+1-i}^{(1)} D_{t-i}=\sum_{i=1}^\infty
  \Pi_{t,t+2-i}^{(1)}D_{t+1-i}. 
 \end{align}
 Since $(W_{1,t}, W_{2,t})$ 
satisfies  SRE \eqref{bivSRE}, then by the uniqueness of
 the solution, we have 
 \begin{lemma}
  Suppose that \eqref{epsilon_moment} and \eqref{logB} are satisfied and let $W_{1,t}$, $W_{2,t}$ be stationary solutions to \eqref{uniSRE1} and \eqref{uniSRE2} respectively.  Then $(W_{1,t},W_{2,t})$ is the stationary solution to \eqref{bivSRE}. 
 \end{lemma}
 Therefore, in the sequel if the stationarity condition, i.e.\ 
 \eqref{epsilon_moment} and \eqref{logB}, is satisfied, we may
 work on these component-wise solutions.

 \section{Main results}

 \subsection{Component-wise tail behavior}
Our aim of this section is to describe the tail behavior of $W_{1,t}$
 and $W_{2,t}$. This is the content of Theorem \ref{thm:mainresult}
 below. We are going to use a fundamental result for one-dimensional SRE formulated below as Theorem \ref{thmunitail}. The statement appeared first in  
 \cite{kesten:1973} as a corollary of a more general result, but the ``right'' proof for the one-dimensional case was given later on by Goldie
 \cite{goldie:1991} with the constants in the tails specified for the
 first time. 
 Consider the SRE; 
 \[
  X_t=Q_t X_{t-1}+R_t,\quad t\in\Z,
 \]
 where $((Q_t,R_t))_{t\in \Z}$ is an $\R^2$-valued i.i.d.\ sequence. 
 The generic random variable (r.v.) of the sequence $((Q_t,R_t))$ is denoted by $(Q,R)$. 
 \begin{theorem}
 \label{thmunitail}
  Assume that the following conditions hold. \\
 1. $Q>0\,a.s.$ and $\log Q$ is non-arithmetic. \\
 2. There is $\alpha>0$ such that $\E Q^\alpha=1$, $\E |R|^\alpha<\infty$
  and $\E Q^\alpha \log^+ Q<\infty$. \\
 3. $\P (Rx+Q=x)<1$ for every $x\in \R$. \\
 Then the equation $X\stackrel{d}{=}QX+R$ has a solution $X$ which is
  independent of $(Q,R)$ and there exist constants $c_+,\,c_-$ such that
  $c_++c_->0$ and 
 \begin{align*}
  \P (X>x) \sim c_+ x^{-\alpha}\quad \mathrm{and}\quad \P (X\le -x) \sim c_-
  x^{-\alpha},\quad x\to\infty.
 \end{align*} 
 The constants $c_+,\,c_-$ are given by 
 \begin{align*}
c_+ = \frac{1}{\alpha m_\alpha}\E[(QX+R)_+^\alpha-(QX)_+^\alpha] \quad
  \mathrm{and}\quad 
c_- = \frac{1}{\alpha m_\alpha}\E[(QX+R)_-^\alpha-(QX)_-^\alpha], 
 \end{align*}
 where $m_\alpha =\E Q^\alpha \log Q >0$. 
 \end{theorem}

Due to stationarity it is enough to consider $W_{1,0}$ and $W_{2,0}$.  
The above result is directly applicable to 
 $W_{2,0}$ but not to $W_{1,0}$. The estimate of the tail of $W_{1,0}$ is more delicate. Due to non-negativity, we may write the stationary solution 
 \eqref{infseriesw1} as 
 \begin{align}
 \label{infseriesw2}
  W_{1,0}=\sum_{i=1}^\infty \Pi_{0,2-i}^{(1)} D_{1-i} = \sum_{i=1}^\infty
  \Pi_{0,2-i}^{(1)} A_{2,1-i} W_{2,-i} +\sum_{i=1}^\infty
  \Pi_{0,2-i}^{(1)} B_{1,1-i}, 
 \end{align}
 and analyze these infinite sums separately. Consider one more SRE;
 \begin{align}\label{uniSRE22}
  \wt W_{1,t} = B_{1,t}+A_{1,t} \wt W_{1,t-1}.
 \end{align}
Its stationary solution is 
 \begin{align*}
  \wt W_{1,0} = \sum_{i=1}^\infty \Pi^{(1)}_{0,2-i} B_{1,1-i}.
 \end{align*}
 This corresponds to the second term in \eqref{infseriesw2}.
 
 Assume that there exist $\alpha_1$ and $\alpha_2$ such that 
 \begin{align} 
  \begin{split} \label{comptailcondi}
  \E A_{1}^{\alpha_1} &=1,\quad \E A_{1}^{\alpha_1} \log^+ A_{1} <\infty\quad
  \mathrm{and} \quad \E B_{1}^{\alpha_1} <\infty, \\
  \E A_{4}^{\alpha_2} &=1,\quad \E A_{4}^{\alpha_2} \log^+ A_{4} <\infty\quad
  \mathrm{and} \quad \E B_{2}^{\alpha_2} <\infty, 
 \end{split}
 \end{align}
 then due to Theorem \ref{thmunitail}, we have
 \begin{align}
  & \P (\wt W_{1,0}>x) \sim c_{1}x^{-\alpha_1},\nonumber \\
  & \P (W_{2,0}>x) \sim c_{2} x^{-\alpha_2}, \label{weq:tails}
 \end{align}
 where positive constants $c_1$ and $c_2$ are given by 
 \begin{align*}
  c_{1} &= \frac{1}{\alpha_1\E A_{1}^{\alpha_1}\log A_{1}} \E [
 (A_{1} \wt W_{1,0} + B_{1})^{\alpha_1} - (A_{1} \wt W_{1,0})^{\alpha_1}
 ], \\
 c_{2} &= \frac{1}{\alpha_2 \E A_{4}^{\alpha_2}\log A_{4}} \E[
 (A_{4} W_{2,0}+B_{2})^{\alpha_2}- (A_{4} W_{2,0})^{\alpha_2}
 ].
 \end{align*}
 Now we are ready to describe the tail behavior of $W_{1,0}$. Its tail index is   equal to $\min\, (\alpha_1,\alpha_2)$. 

 \begin{theorem}
 \label{thm:mainresult}
  Consider the bivariate SRE \eqref{bivSRE} such that 
  $\log A _1$ and $\log A_4$ are non-arithmetic. Assume that \eqref{comptailcondi} holds and $\E A_{2}^{\min\,
  (\alpha_1,\alpha_2)}<\infty$. Then the stationary solution $\bfW_t$ satisfies
 \begin{align*}
  \P (W_{1,0}>x) \sim \Bigg \{
\begin{array}{ll}
\ov c_{1} x^{-\alpha_1} & \mathrm{if}\ \alpha_1
   < \alpha_2 \\
\wt c_{1} x^{-\alpha_2} &\mathrm{if}\
   \alpha_1 > \alpha_2, 
\end{array}
 \end{align*}
 where 
  \begin{align}
 \label{rep:constant2}
  \ov c_{1} &= \frac{2}{\alpha_1}\E\big[
 (D_0+A_{1,0} W_{1,-1})^{\alpha_1} - (A_{1,0} W_{1,-1})^{\alpha_1}
 \big], \\
 \label{rep:constant1}
  \wt c_{1} &=c_2\, \E\Big( \lim_{s\to\infty} \sum_{i=1}^s \Pi_{0,2-i}^{(1)}
 A_{2,1-i} \Pi_{-i,1-s}^{(4)}
 \Big)^{\alpha_2}.
 \end{align}
 \end{theorem}

 Note that the limit in \eqref{rep:constant1} has a somewhat strange form,
 and it seems difficult to write it as just an infinite sum. However,
 its convergence is guaranteed in the proof. 

 \begin{proof}
  Stationarity condition \eqref{epsilon_moment} follows from \eqref{comptailcondi}. Indeed, the functions
  $g_i(h)=\E A_i^h,\,i=1,4$ are convex and so  there exist 
  $\alpha<\min\,(\alpha_1,\alpha_2)$ such that $\E A_i^\alpha<1$. Hence we
  may work on the stationary version. 
 
  Suppose $\alpha_1>\alpha_2$. 
  Observe that in \eqref{infseriesw2}, each term of the first infinite
  sum includes $W_{2}$, $\E A_{1}^{\alpha_2},\,\E
  A_{2}^{\alpha_2}<\infty$,
  and the second sum equals in distribution to $\wt W_{1,0}$. Then
  since 
 \begin{align*}
  \frac{\P (\wt W_{1,0}>x)}{\P(W_{2,0}>x)} \sim
  \frac{c_{1,+}}{c_{2,+}} x^{\alpha_2-\alpha_1} \to 0,\quad as \quad
  x\to\infty, 
 \end{align*}
 it suffices to consider the sum
 \begin{align}
 \label{eq:expressX}
  \wt X = \sum_{i=1}^\infty \Pi_{0,2-i}^{(1)} A_{2,1-i} W_{2,-i}
  =\sum_{i=1}^s \Pi_{0,2-i}^{(1)} A_{2,1-i}W_{2,-i}+ \sum_{i=s+1}^\infty
  \Pi_{0,2-i}^{(1)}A_{2,1-i} W_{2,-i} =: \wt X_s +\wt X^s. 
 \end{align}
Indeed, $\wt X\geq A_{2,0} W_{2,-1}$ and so 
\begin{equation}\label{heavyX}
\P (\wt X>x)\geq cx^{-\a _2}
\end{equation}
for a strictly positive $c$. Therefore, we can invoke the property of
  dependent summands of regularly varying r.v.'s (
Lemma B.6.1 of \cite{buraczewski:damek:mikosch:2016}) in  order to obtain 
 \[
  \P (W_{1,0}>x) = \P (\wt X+\wt W_{1,0}>x) \sim \P (\wt X>x)+\P (\wt W_{1,0}>x)
  \sim \P (\wt
  X>x). 
 \]
 We start with $\wt X_s$ and apply the induction to
  $W_{2,-i}$ in $\wt X_s$. Since 
 \begin{align*}
  W_{2,t} &= A_{4,t}W_{2,t-1}+B_{2,t} \\
         &= A_{4,t} A_{4,t-1} W_{2,t-2}+A_{4,t}B_{2,t-1}+B_{2,t} \\
         &= A_{4,t}A_{4,t-1}\cdots A_{4,t-j+1}W_{2,t-j} +
  \sum_{k=1}^{j-1}A_{4,t}\cdots A_{4,t-k+1}B_{2,t-k}+B_{2,t} \\
  &= \Pi_{t,t-j+1}^{(4)} W_{2,t-j}
  +\sum_{k=0}^{j-1}\Pi_{t,t-k+1}^{(4)} B_{2,t-k}, 
 \end{align*}
 we change indices $(t,j)\to (i,s)$ such that $t=-i$ and $s=j-t$ to obtain 
 \begin{align}
 \label{eq:expressW2}
  W_{2,-i}= \Pi_{-i,1-s}^{(4)}
  W_{2,-s}+\sum_{k=0}^{s-i-1}\Pi_{-i,1-i-k}^{(4)} B_{2,-i-k}. 
 \end{align}
 Substitution of the above into $\wt X_s$ yields 
 \begin{align}
 \label{eq:expressXs}
  \wt X_s &= \sum_{i=1}^s \Pi_{0,2-i}^{(1)} A_{2,1-i}\Big(
 \Pi_{-i,1-s}^{(4)}W_{2,-s}+ \sum_{k=0}^{s-i-1}\Pi_{-i,1-i-k}^{(4)} B_{2,-i-k}
 \Big) \\
 &= \sum_{i=1}^s \Pi_{0,2-i}^{(1)} A_{2,1-i}\Pi_{-i,1-s}^{(4)}W_{2,-s} +
  \sum_{i=1}^s \Pi_{0,2-i}^{(1)} A_{2,1-i}\sum_{k=0}^{s-i-1}
  \Pi_{-i,1-i-k}^{(4)} B_{2,-i-k} \nonumber \\
 &= \wt X_{s,1} + \wt X_{s,2}. \nonumber 
 \end{align}
Now we are going to prove that $\E \wt X_{s,2}^{\a _2}<\infty$. 
Recall that $(\bfA_t,\bfB_t)$ are i.i.d. and hence 
$\Pi_{0,2-i}^{(1)} A_{2,1-i}$ and $\sum_{k=0}^{s-i-1}
  \Pi_{-i,1-i-k}^{(4)} B_{2,-i-k},\,i=1,2,\ldots,s$ are independent. 
Moreover, $\E A_{4}^{\alpha_2}=1$,
 $\E A^{\alpha_2}_{1}<1$, $\E A_{2}^{\alpha_2}<\infty$
  and $\E B_{i}^{\alpha_i}<\infty$. 
 By the Minkowski inequality 
 \begin{align*}
  \E \wt X_{s,2}^{\alpha_2} &= \E 
 \big(
 \sum_{i=1}^s \Pi_{0,2-i}^{(1)} A_{2,1-i} \sum_{k=0}^{s-i-1}
  \Pi_{-i,1-i-k}^{(4)} B_{2,-i-k}
 \big)^{\alpha_2} \\
 & \le \Big[
 \sum_{i=1}^s \Big\{
 (\E A_{1}^{\alpha_2})^{i-1} \E A_{2}^{\alpha_2} \E
  B_{2}^{\alpha_2} (s-i)^{\alpha_2}
 \Big\}^{1/\alpha_2}
 \Big]^{\alpha_2}<\infty
 \end{align*}
 for $\alpha_2> 1$ 
and for $\alpha_2\le 1$ by
  the triangle inequality 
 \begin{align*}
  \E \wt X_{s,2}^{\alpha_2} &= \sum_{i=1}^s \E
  (\Pi_{0,2-i}^{(1)})^{\alpha_2}\E A_{2,1-i}^{\alpha_2} \E \big(
 \sum_{k=0}^{s-i-1} \Pi_{-i,1-i-k}^{(4)} B_{2,-i-k}\big)^{\alpha_2} \\
 & \le \E A_{2}^{\alpha_2}\E B_{2}^{\alpha_2} \sum_{i=1}^s (\E
  A_{1}^{\alpha_2})^{i-1}(s-i)<\infty .
 \end{align*}
Hence we have 
 \[
  \limsup_{x\to\infty} \frac{\P (\wt X_{s,2}>x)}{\P (W_{2,0}>x)}=0 
 \]
 for fixed $s$. 
 Since all the terms are positive
\[
 \P (\wt X_{s,1}>(1+\varepsilon)x)\le \P (\wt X>x) \le \P (\wt
  X_{s,1}>(1-\varepsilon)x) +\P (\wt X_{s,2}>\frac{\varepsilon}{2}x)
  +\P (\wt X^s >\frac{\varepsilon}{2}x)
\] 
 for some $\varepsilon \in (0,1)$, it follows from regular variation of
  $W_{2,0}$ that for $s\ge 1$,
 \begin{align}
 \begin{split}
  (1+\varepsilon)^{-\alpha_2} w_s
  \le \liminf_{x\to\infty} & \frac{\P (\wt X>x)}{\P (W_{2,0}>x)} 
 \label{estimate} 
 \le 
  \limsup_{x \to \infty} \frac{\P(\wt X>x)}{\P(W_{2,0}>x)} \le (1-\varepsilon)^{-\alpha_2} w_s \\ 
   & + \limsup_{x\to \infty} \frac{\P (\wt X_{s,2}>
  \frac{\varepsilon}{2}x)}{\P (W_{2,0}>x)}+ \limsup_{x\to \infty}
  \frac{\P (\wt X^s >\frac{\varepsilon}{2}x)}{\P (W_{2,0}>x)},
 \end{split}
 \end{align}
 where 
 \begin{align*}
  w_s:=\E 
 \big(
 \sum_{i=1}^s \Pi_{0,2-i}^{(1)}\, A_{2,1-i}\, \Pi_{-i,1-s}^{(4)}
 \big)^{\alpha_2}.
 \end{align*}
 Hence, in order to obtain the result, it suffices to show that 
 \begin{align}
 \label{limitto01}
  \lim_{s\to\infty} \limsup_{x\to \infty} \frac{\P(\wt X^s>x)
  }{\P(W_{2,0}>x)} =0. 
 \end{align}
However, by Markov inequality together with conditioning, it follows
  that 
\begin{align*}
 \frac{\P(\wt X^s >x)}{\P(W_{2,0}>x)} &= \frac{\P(\sum_{i=s+1}^\infty
 \Pi_{0,2-i}^{(1)} A_{2,1-i} W_{2,-i}>x 
)}{\P(W_{2,0}>x)} \\
 & \le \sum_{i=1}^\infty
 \frac{ \P\big( \Pi^{(1)}_{0,2-(s+i)} A_{2,1-i} W_{2,-(s+i)}>x
 i^{-\mu}/\zeta(\mu)\big)}{\P(W_{2,0}>x)} \\
 & = \sum_{i=1}^\infty \E \Big[
 \frac{\P(G_i W_{2,-(i+s)} >x \mid G_i)}{\P(W_{2,0}>x)}
 \Big], 
\end{align*}
 where $G_i=\zeta(\mu) i^\mu \Pi_{0,2-(s+i)}^{(1)}A_{2,1-(s+i)}$ with
  $\mu>1$ and $\zeta(\mu)$ is the zeta function. 
Then 
\begin{align*}
 \limsup_{x\to \infty} \frac{\P(\wt X^s >x)}{\P(W_{2,0}>x)} &=
 \sum_{i=1}^\infty \E \Big[
 \limsup_{x\to \infty}
 \frac{\P(G_i W_{2,-(i+s)} >x \mid G_i)}{\P(W_{2,0}>x)}
 \Big] \\
 &\le  c \sum_{i=1}^\infty \E G_i^{\alpha_2} \\
 & = c \sum_{i=1}^\infty \E(\Pi_{0,2-(s+i)}^{(1)})^{\alpha_2}\E
 A_{2,1-(s+i+1)}^{\alpha_2} \zeta(\mu)i^{\alpha_2\mu} \\
 & \le c' (\E A_1^{\alpha_2})^s \sum_{i=1}^\infty (\E A_1^{\alpha_2})^i
 i^{\alpha\mu}, 
\end{align*}
where $c$ and $c'$ are some positive constants. Since $\E A_1^{\alpha_2}<1$
  so that $\sum_{i=1}^\infty (\E A_1^{\alpha_2})^i
  i^{\alpha\mu}<\infty$, we take $s\to\infty$ and obtain
  \eqref{limitto01}. 
Now, as before, if $\a _2 \leq 1$ then 
$$
w_s\leq \sum _{i=1}^{\infty}(\E A_1^{\a _2})^{i-1}\E A_2 ^{\a _2}<\infty$$ and if $\a _2>1$ then
 $$
w_s^{1\slash \a _2}\leq \sum _{i=1}^{\infty}(\E A_1^{\a _2})^{(i-1)\slash \a _2}(\E A_2 ^{\a _2})^{1\slash \a _2}<\infty .$$
Moreover, in view of \eqref{heavyX} and \eqref{estimate} there is
  $c''>0$ such that for every $s$, $w_s\geq c''$. 
Hence taking a converging subsequence $w_{s_k} $ of $w_s$ we obtain
\begin{equation}
\lim_{x \to \infty} \frac{\P(\wt X>x)}{\P(W_{2,0}>x)}=\lim _{k\to \infty }w_{s_k} =:w>0,
\end{equation}
which, in particular, proves that $\lim _{s\to \infty }w_{s} =w$ because we have the same limit for any converging subsequence. Finally, we obtain
 \[
  \P(W_{1,0}>x) \sim \lim_{s\to\infty} \E\big(
 \sum_{i=1}^s \Pi_{0,2-i}^{(1)} A_{2,1-i} \Pi_{-i,1-s}^{(4)}
\big)^{\alpha_2} \P (W_{2,0}>x). 
 \]  

Suppose now that $\alpha_1<\alpha_2$. Observe that
\begin{align*}
 W_{1,0} & = D_0 + \sum_{i=1}^\infty A_{1,0}\cdots A_{1,1-i} D_{-i} \\
 & = D_0 + A_{1,0} D_{-1} + \sum_{i=2}^\infty A_{1,0}\cdots
 A_{1,1-i}D_{-i} \\
 & = D_0 + A_{1,0}\Big(
 D_{-1} + \sum_{i=1}^\infty A_{1,-1}\cdots A_{1,-i}D_{-i-1}
 \Big) \\
 & = D_0 + A_{1,0} W_{1,-1}, 
\end{align*}
Then $W_{1,-1}$ has the same law as $W_{1,0}$ and is independent of $A_{1,0}$. We are going to use Theorem 2.3 of Goldie \cite{goldie:1991}. We will have to prove that 
\[
 I=\int_{0}^\infty |\P (W_{1,-1} >x)- \P (A_{1,0} W_{1,-1}>x)|x^{\alpha_1-1}\d x <\infty. 
\]
Then Theorem 2.3 of \cite{goldie:1991} implies that 
\[
 \lim_{t\to\infty} \P (W_{1,-1}>x)\, x^{\alpha_1} = C_+
\]
with 
\[
 C_+=\frac{1}{\E A_{1}^{\alpha_1}\log A_{1}}\int_0^\infty
  (\P(W_{1,-1}>x)-\P(A_{1,0} W_{1,-1}>x))x^{\alpha_1-1}\d x. 
\]
In view of Lemma 9.4 in \cite{goldie:1991}, 
\begin{align*}
 I &= \int_{0}^\infty |\P(W_{1,-1}>x)-\P(A_{1,0} W_{1,-1} >x)|\,x^{\alpha_1-1}\d x \\
 &= \int_{0}^\infty | \P(D_0+A_{1,0} W_{1,-1} >x) -
 \P(A_{1,0} W_{1,-1}>x) |\, x^{\alpha_1-1}\d x \\
 &= \frac{1}{\alpha_1} \E\Big[
 (D_0+A_{1,0} W_{1,-1} )^{\alpha_1}- (A_{1,0} W_{1,-1})^{\alpha_1} 
\Big] = \E A_{1}^{\alpha_1}\log A_{1}\cdot C_+.
\end{align*}
We have to prove that $0<I<\infty$. The first inequality is obvious,
  since all the variables are positive and $D_0+A_{1,0} W_{1,-1}
> A_{1,0}W_{1,-1}$ with positive probability. If
  $\alpha_1 \le 1$, then
 \[
  I\le \frac{1}{\alpha_1}\E D^{\alpha_1}<\infty. 
 \]
If $\alpha_1>1$, then 
\begin{align*}
 I 
  & \le \E\big[
 (D_0+A_{1,0} W_{1,-1})^{\alpha_1-1}D_0 
\big],
 \end{align*}
where we have used the fact that $b^\kappa -a^\kappa =\int_a^b \kappa
 r^{\kappa-1} \d r \le \kappa b^{\kappa-1}(b-a)$ for any $\kappa>1$ and
 $b \ge a \ge 0$. Further we have 
 \[
 I \le \max (2^{\alpha_1-2},1)\, \{\E D_0^{\alpha_1} 
 +\E (A_{1,0} W_{1,-1})^{\alpha_1-1}D_0\},
 \] 
 where we use the Minkowski's inequality and subadditivity of concave functions depending on
  whether $\alpha_1 > 2$ or $\alpha_1\le 2$, and we need to prove that
 \begin{align*}
  & \E(A_{1,0} W_{1,-1})^{\alpha_1-1}D_0 \\ 
  &= \E(A_{1,0} W_{1,-1})^{\alpha_1-1}(B_{1,0}+A_{2,0}
  W_{2,-1}) \\
  & = \E(A_{1,0} W_{1,-1})^{\alpha_1-1}B_{1,0}+
  \E(A_{1,0} W_{1,-1})^{\alpha_1-1} A_{2,0}W_{2,-1}<\infty. 
 \end{align*}
Since $(A_{1,0},B_{1,0})$ and $W_{1,-1}$ are independent and since 
$A_{1,0}^{\alpha_1-1}A_{2,0}$ and $(W_{1,-1})^{\alpha_1-1}W_{2,-1}$
  are independent, what we need to show is 
 \[
  \E [A_{1,0}^{\alpha_1-1}B_{1,0}] \E W_{1,0}^{\alpha_1-1} +
  \E[A_{1,0}^{\alpha_1-1}A_{2,0}]\E [(W_{1,-1})^{\alpha_1-1}W_{2,-1}]<\infty.
 \] 
By H\"older's inequality, for $X=B_{1,0}$ or $A_{2,0}$ we have 
\[
 E[A_{1,0}^{\alpha_1-1}X] \le (\E A_{1,0}^{p(\alpha_1-1)})^{1/p}(\E X^q)^{1/q} =
  (\E A_{1,0}^{\alpha_1})^{1/p}(\E X^{\alpha_1})^{1/q}<\infty 
\]
with $p=\alpha_1/(\alpha_1-1)$ and $q=\alpha_1$. This together with
$\E W_{1,0}^{\alpha_1-1}<\infty$ shows
  $\E[A_{1,0}^{\alpha_1-1}B_{1,0}]\E W_{1,0}^{\alpha_1-1}<\infty$. 
Therefore it suffices to see $\E[(W_{1,-1})^{\alpha_1-1}W_{2,-1}]<\infty$.
  For any positive $\varepsilon<\alpha_2-\alpha_1$ we can deduce that $\E
  W_{2}^{\alpha_1+\varepsilon}<\infty$ by the property
  \eqref{weq:tails}. Let $q=\alpha_1+\varepsilon$ and its H\"older
  conjugate $p=q/(q-1)$. Then
 \[
  \E \big[
 (W_{1,-1})^{\alpha_1-1}W_{2,-1}
 \big] \le \big(
 \E W_{1,0}^{p(\alpha_1-1)}
 \big)^{1/p} (\E W_{2,0}^q)^{1/q}, 
 \] 
 where $\E W_2^q<\infty$. 
 Notice that $\beta:=p(\alpha_1-1)<\alpha_1$ because $p$ is decreasing
  in $q$ and $p=\alpha_1/(\alpha_1-1)$ implies $q=\alpha_1$. Here we may
  choose $\varepsilon$ such that $1<\beta$.
 Then by convexity $\E A_1^\beta<1$ and moreover, since
  $1<\beta<\alpha_1<\alpha_2$, for $k\in \Z$ 
\[
 \E D_k^\beta \le \big(
 (\E B_{1,k}^\beta)^{1/\beta}+(\E (A_{2,k}W_{2,k-1})^{\beta} )^{1/\beta}
 \big)^\beta <\infty.
\] 
Now, notice that $\E W_{1,0}^\beta =\E
  (\sum_{i=0}^\infty \Pi_{0,1-i}^{(1)} D_{-i} )^\beta $ 
and so by 
Minkowski's inequality, it is enough to prove that 
\[
 \sum_{i=0}^\infty \big(
 \E (\Pi_{0,1-i}^{(1)} D_{-i})^\beta
 \big)^{1/\beta} = \sum_{i=0}^\infty (\E A_{1}^\beta)^{i/\beta}
  (\E D^\beta)^{1/\beta} <\infty,
\]
which holds, because $\E A_{1}^\beta <1$. 
\end{proof}
\subsection{Regular variation}
Now we are going to study regular variation of the strictly stationary time series $(\bfW_t)=((W_{1,t},W_{2,t})')$. As before, we distinguish two cases: 
$\alpha_1<\alpha_2$ and $\alpha_1>\alpha_2$. In the first case the tail indices of $(W_{i,t}),\,i=1,2$ are distinct so we consider the components separately.
 
Let us start with discussing regular variation. A univariate time series is said to be regularly varying if its
finite-dimensional distributions are such. The latter is meant in the sense of
 \eqref{regular}. 
  More precisely, let $\bfX$ be an $h$-dimensional r.v.\ It is called {\it multivariate regularly
 varying with index $\alpha$} if 
 \begin{align}
 \label{eq:mrv1}
  \frac{\P(|\bfX|>ux,\,\bfX/|\bfX|\in \cdot)}{\P(|\bfX|>x)}
  \stackrel{v}{\to} u^{-\alpha}\P(\bfTh\in \cdot),\quad u>0,
 \end{align}
 where $\stackrel{v}{\to}$ denotes vague convergence and $\bfTh$ is a
 random vector on the
 unit sphere $\bbs^{h-1}=\{\bfx\in \R^h \mid |\bfx| =1 \}$.
\footnote{Note that in the univariate case, 
we say that a positive measurable function $f(x)$ is regularly varying with index
 $\rho$ if $\lim_{x\to \infty}f(cx)/f(x)=c^\rho,\,c>0$. Moreover,
 r.v. $X$ is said to be regularly varying with index $\alpha>0$ if
 $f(x)=\P(|X|>x)$ is regularly varying with index $-\alpha$, see \cite[p.273]{buraczewski:damek:mikosch:2016}. A
 similar definition is used for the multivariate case, see
 \cite[p.279]{buraczewski:damek:mikosch:2016}.   
}
 Its distribution is called the spectral measure of the regularly varying vector 
 $\bfX$. This type of approach to determine the tail behavior of a
 univariate strictly stationary series was introduced by Davis and
Hsing \cite{davis:hsing:1995} and was 
used by e.g.\  
 Mikosch and St\u{a}ric\u{a} \cite{mikosch:starica:2000}. See also \cite{buraczewski:damek:mikosch:2016}, page 273.
 

To characterize the regular variation of $(W_{i,t}),\,i=1,2$,  we 
use the following notation for $h\ge 1,\,i=1,2,\,j_1=1$ and $j_2=4$: 
\begin{align}
 \label{eq:wpr}
 \bfW_{i,h} =(W_{i,1},\ldots,W_{i,h}),\quad
 {\bf \Xi}_h^{(i)}=(\Pi_{1}^{(j_i)},\ldots,\Pi_{h}^{(j_i)}),\quad
 \text{and}\quad \bfR_{i,h}=(R_{i,1},\ldots,R_{i,h}), 
\end{align}
 where $\Pi_\ell^{(i)}=\Pi_{\ell,1}^{(i)}$ and  
\begin{align*}
 R_{1,t} = \sum_{i=0}^{t-1} \Pi_{t,t+1-i}^{(1)} 
 D_{t-i}
 \quad \mathrm{and}\quad 
 R_{2,t} = \sum_{i=0}^{t-1} \Pi_{t,t+1-i}^{(4)} 
 B_{2,t-i}
 ,\,t\ge 1. 
\end{align*}
\begin{lemma}
 \label{lem:differenttail}
 Suppose that $\alpha_1<\alpha_2$ and the conditions of Theorem \ref{thmunitail} are satisfied. Then strictly stationary series $(W_{1,t})$ and $(W_{2,t})$ are
 regularly varying with indices $\alpha_1$ and $\alpha_2$
 respectively and the spectral measures for finite dimensional vectors
 $(W_{i,1},\ldots,W_{i,h}),\,h\ge1,\,i=1,2$ are 
 \begin{align*}
  \P(\bfTh_{i,h}\in \cdot)= \frac{\E |{\bf \Xi}_h^{(i)}|^{\alpha_i}\,
  {\bf 1}\, ( {\bf \Xi}_h^{(i)}/|{\bf \Xi}_h^{(i)}| \in \cdot
  )}{\E|{\bf \Xi}_h^{(i)}|^{\alpha_i} }, 
 \end{align*}
 where r.v.'s $\bfTh_{i,h}$ take values on $\bbs^{h-1}$. 
\end{lemma}

 \begin{proof}
The proofs for both series are very similar, so we give the proof only for $(W_{1,t})$. 
  Since $W_{1,t}=\Pi_{t}^{(1)}W_{1,0}+R_{1,t},\,t=1,\ldots,h$ by
  induction, we have a representation $\bfW_{1,h}={\bf \Xi}_h^{(1)}W_{1,0}+\bfR_{1,h} $
  where both vectors ${\bf \Xi}_h^{(1)}$ and $\bfR_{1,h}$ have the moment of
  order $\alpha_1$. Due to the multivariate Breiman's lemma 
 \cite[Lemma C.3.1 (1)]{buraczewski:damek:mikosch:2016}, 
 we have 
 \[
  \lim_{x\to\infty}\frac{\P(x^{-1}\bfW_{1,h} \in \cdot)}{\P(W_{1}>x)}
  = \lim_{x\to\infty} \frac{\P(x^{-1}{\bf \Xi}_h^{(1)}W_{1}\in\cdot)}{\P(W_{1}>x)},
 \]
so that it suffices to study the regular variation of
  ${\bf \Xi}_h^{(1)}W_{1}$. Moreover, applying Breiman's lemma again, we obtain
 as $y\to \infty$, 
\begin{align*}
 \P\Big( |W_{1}{\bf \Xi}_h^{(1)}| >xy,\,
 \frac{{\bf \Xi}_h^{(1)}W_{1}}{|{\bf \Xi}_h^{(1)}W_{1}|}\in \cdot
\Big) &= \P\Big( W_{1}|{\bf \Xi}_h^{(1)}|\, {\bf 1}\, ( 
{\bf \Xi}_h^{(1)}/|{\bf \Xi}_h^{(1)}| \in \cdot
) >xy 
\Big) \\
 &\sim \E |{\bf \Xi}_h^{(1)}|^{\alpha_1}\, 
{\bf 1}\, ( 
 {\bf \Xi}_h^{(1)}/|{\bf \Xi}_h^{(1)}| \in \cdot
 ) x^{-\alpha_1}\P (W_{1}>y) 
\end{align*}
 and 
 \begin{align*}
  \P(|\bfW_{1,h}| >xy) \sim
  \E|{\bf \Xi}_h^{(1)}|^{\alpha_1} x^{-\alpha_1}\P(W_{1}>y).
 \end{align*}
Hence the conclusion follows. By the same logic the time series $(W_{2,t})$ is shown to be regularly
  varying with index $\alpha_2$ as a series. 
 \end{proof}
Secondly, we study the case $\alpha_2<\alpha_1$ where both component
processes have the same tail index $\alpha_2$, so we consider
a bivariate time series.
However, it is more convenient to modify slightly the definition of
regular variation i.e.\ to adopt the version better to the bivariate case as done
by Basrak and Segers \cite{basrak:segers:2009}.
 
 An $\R^d$-valued strictly stationary time
 series $(\bfX_t)$ is {\em \regvary\ 
with index $\alpha>0$} 
 if the following limits in \ds\ exist
\beam\label{eq:jan6a}
\P(x^{-1} (\bfX_0,\ldots,\bfX_h)\in \cdot\mid |\bfX_0|>x )\stw
\P( (\bfY_0,\ldots,\bfY_h)\in \cdot)\,,\qquad \xto\,,
\eeam 
where $\stw$ denotes weak convergence. The limit vector $(\bfY_0,\ldots,\bfY_h)$ has the same \ds\ as 
$|\bfY_0|(\bfTh_0,\ldots,\bfTh_h)$, where the \ds\ of $|\bfY_0|$ is given by
$\P(|\bfY_0|>y)=y^{-\alpha}$, $y>1$, and $|\bfY_0|$ and 
$(\bfTh_0,\ldots,\bfTh_h)$ are independent. The distribution of $\bfTh_0$ is the spectral measure of $\bfX_0$ and $(\bfTh_t)_{ t\ge 0}$ is the
spectral process.
Notice that $\bfTh_t,\,t\neq 0$ is not always on $\bbs^{d-1}$. 
 The equivalence of
definitions \eqref{eq:mrv1} and \eqref{eq:jan6a} is proved in
\cite{basrak:segers:2009} but \eqref{eq:jan6a} is usually easier to handle see e.g.\ \cite{matsui:mikosch:2016}. 



\begin{proposition}
\label{prop:fidm2}
 Assume that $\alpha_1>\alpha_2>0$ and that the conditions of Theorem \ref{thmunitail} are satisfied.  Then the
 bivariate strictly stationary series $\bfW_t=(W_{1,t},W_{2,t})'$ is
 regularly varying with index $\alpha_2$ in the sense of
 \eqref{eq:jan6a} and 
 \begin{align}
  \label{eq:spectralpros}
  \P\big(x^{-1}(\bfW_1,\ldots,\bfW_h)\in\cdot \mid
 |\bfW_0|>x \big) \stackrel{w}{\to} \P\big(
 Y ({\bf\Pi}_1\bfTh_0,\ldots,{\bf\Pi}_h\bfTh_0)\in \cdot
 \big),\quad h\ge 1, 
 \end{align}
where $\P(Y>x)=x^{-\alpha_2},x>1$, $Y$ is independent of $(\bfTh_0,
 {\bf\Pi}_1,\ldots,{\bf\Pi}_h)$. The distribution of  
 $\bfTh_0$ has the spectral measure of $\bfW_0$. Here $\bfTh_0$ and
 $({\bf\Pi}_1,\ldots,{\bf\Pi}_h)$ are also independent. 
\end{proposition}
\begin{remark}
(i) The same characterization has been examined in 
\cite{matsui:mikosch:2016}.  \\
(ii) As seen above the law of $\bfTh_0$ is the most important in our
 case. Usually analytical expressions for the law of $\bfTh_0$ are not
 available. However, since the law of $\bfTh_0$ satisfies a certain
 invariant relation by index $\alpha_2$ and $\bfA$, we could simulate
 $\bfTh_0$ (see Proposition 5.1 in \cite{basrak:segers:2009}). We do not
 pursuit this here and only make the following remark. Although our
 setting for SRE are different from usual assumption, i.e.\ $\bfA$ is
 triangular, once bivariate regular variation for $\bfW$ has been proved,
 we can use the method in \cite{basrak:segers:2009}. (We checked the
 assumptions in Section 5 of \cite{basrak:segers:2009} and did not find
 any problems for our case.)
\end{remark}
\begin{proof}
 We show that for every $\bfy=(y_1,y_2)\in \R ^2 $,
 \begin{align}
 \label{def:rv1}
  \lim_{x\to\infty} \frac{\P(\bfy'\bfW_0>x)}{\P(W_{2}>x)}=w(\bfy)\quad
  \mathrm{exists}, 
 \end{align}
$$w(\bfy)>0\quad  \mbox{if}\quad \bfy \in[0,\infty)^2\setminus \{\bf0\}$$
and $$w(\bfy)=0\quad  \mbox{if}\quad  \bfy \in (-\infty , 0)^2.$$
Then by Boman and Lindskog \cite{boman:lindskog:2009}, see also \cite[Appendix C]{buraczewski:damek:mikosch:2016},
 we may conclude the regular variation of $\bfW_0$ in the
 sense of \eqref{eq:mrv1}. 
 In view of the proof of Theorem \ref{thmunitail}
 (\eqref{eq:expressX}, \eqref{eq:expressW2} and \eqref{eq:expressXs}), we recall that
 \begin{align*}
  W_{1,0}=\wt X+\wt W_{1,0} = \wt X_{s,1} + \wt X_{s,2} + \wt X^s+\wt
  W_{1,0}
  \quad \mathrm{and} \quad W_{2,0}=
  \Pi_{0,1-s}^{(4)}W_{2,-s}+\sum_{i=0}^{s-1}\Pi_{0,1-i}^{(4)}B_{2,-i}.  
 \end{align*}
Given $\varepsilon\in(0,1)$, let
\begin{align*}
M_{x,s}=& \{\,y_1\wt X_{s,1}+y_2\Pi_{0,1-s}^{(4)}W_{2,-s}> x\,\}, \\
M'_s =& \Big\{\,y_1(\wt W_{1,0}+\wt X_{s,2}+\wt X^s)+y_2
  \sum_{i=0}^{s-1}\Pi_{0,1-i}^{(4)}B_{2,-i}<-\varepsilon x \,\Big\}, \\
M''_s =&\Big\{\,y_1(\wt W_{1,0}+\wt X_{s,2}+\wt X^s)+y_2
  \sum_{i=0}^{s-1}\Pi_{0,1-i}^{(4)}B_{2,-i}>\varepsilon x \,\Big\}.
\end{align*}
Then
 \begin{equation*}
M_{(1+\varepsilon )x,s}\setminus M'_s \subset \{\,y_1
  W_{1,0}+y_2 W_{2,0}>x\,\} \subset M_{(1-\varepsilon )x,s}\cup  M''_s. 
\end{equation*}
First we notice that
\begin{align*}
 y_1\wt X_{s,1}+y_2 \Pi_{0,1-s}^{(4)}W_{2,-s}&=\Big( y_1 \sum_{i=1}^s
 \Pi_{0,2-i}A_{2,1-i}\Pi_{-i,1-s}^{(4)}+
y_2 \Pi_{0,1-s}^{(4)}\Big)W_{2,-s}\\
&=:J (\bfy;s)W_{2,-s}.
\end{align*}
And so 
 \begin{equation*}
  \lim_{x\to\infty} \frac{\P(y_1\wt
  X_{s,1}+y_2\Pi_{0,1-s}^{(4)}W_{2,-s}>x)}{\P(W_{2}>x)} =\E J
  (\bfy;s)^{\a _2}
\1(J (\bfy;s)>0)=:w_s(\bfy).
\end{equation*}
Moreover, $w_s(\bfy)$ is bounded independently of $s$. Indeed, for $\a _2> 1$, by the Minkowski inequality, we have
 \begin{align*}
 \E J (\bfy;s)^{\alpha_2} 
 &\le \Big\{
 |y_1| \sum_{i=1}^s \Big(\E\big(
 \Pi_{0,2-i}A_{2,1-i}\Pi_{-i,1-s}^{(4)}\big)^{\alpha_2}\Big)^{1/\alpha_2}+
 |y_2| \Big(\E\big(\Pi_{0,1-s}^{(4)}\big)^{\alpha_2}\Big)^{1/\alpha_2}
 \Big\}^{\alpha_2} \\
 & \le \Big\{
 |y_1| \E A_{2}^{\alpha_2} \sum_{i=1}^s (\E
  A_{1}^{\alpha_2})^{(i-1)/\alpha_2} +|y_2|
 \Big\}^{\alpha_2} <\infty. 
 \end{align*}
For $\alpha_2\le 1$ 
 \[
  \E J (\bfy;s)^{\alpha_2} \le |y_1|^{\alpha_2}\E A_{2}^{\alpha_2} \sum_{i=1}^s (\E A_{1}^{\alpha_2})^{i-1}+|y_2|^{\alpha_2}<\infty,
 \]
 which follows from the triangle inequality.
Now we have
\begin{align*}
\limsup_{x\to\infty} \frac{\P(\bfy' 
  \bfW_0>x)}{\P(W_{2}>x)} &\le \limsup_{x\to\infty}\frac{\P (M_{x(1-\varepsilon),s})}{\P(W_{2}>x)}+\limsup_{x\to\infty}\frac{\P (M ''_s)}{\P(W_{2}>x)}\\
&\le (1-\varepsilon )^{-\alpha_2} 
  w_s(\bfy) + \limsup_{x\to\infty} \frac{\P(y_1\wt X^s
  >\varepsilon x)}{\P (W_{2}>x)}
 \end{align*}
and
\begin{align*}
\liminf_{x\to\infty} \frac{\P(\bfy'
  \bfW_0>x)}{\P(W_{2}>x)} &\ge \lim_{x\to\infty}\frac{\P (M_{x(1+\varepsilon),s})}{\P(W_{2}>x)}-\limsup_{x\to\infty}\frac{\P (M '_s)}{\P(W_{2}>x)}\\
& \ge (1+\varepsilon )^{-\alpha_2}
  w_s(\bfy) - \limsup_{x\to\infty} \frac{\P(|y_1|\wt X^s
  >\varepsilon x)}{\P (W_{2}>x)}. 
 \end{align*}
It follows from \eqref{limitto01} that the last term in these inequality vanishes. 

Then we take a subsequence $s_k$ such that $w_{s_k}(\bfy)$ is convergent and we obtain
\begin{equation*}
(1+\varepsilon )^{-\alpha_2}
  \lim _{k\to \infty}w_{s_k}(\bfy)\le \liminf_{x\to\infty} \frac{\P(\bfy'
  \bfW_0>x)}{\P(W_{2}>x)}\le \limsup_{x\to\infty} \frac{\P(\bfy'
  \bfW_0>x)}{\P(W_{2}>x)}\le (1-\varepsilon )^{-\alpha_2}
  \lim _{k\to \infty}w_{s_k}(\bfy)
\end{equation*}
and so letting $\varepsilon \to 0$ we obtain \eqref{def:rv1}. Moreover, if $\bfy =(y_1,y_2)\in [0,\infty )^2\setminus \{ \bfO\}$ then 
\begin{equation*}
\bfy '\bfW \ge y_1\wt X + y_2 W_{2}\ge \max \{ y_1\wt X , y_2 W_{2}\} 
\end{equation*}
and since both $\wt X, W_{2}$ are regularly varying with index $\a _2$,
$\lim _{k\to \infty}w_{s_k}(\bfy)>0$.


 Next we see \eqref{eq:spectralpros}. By induction \[
 \bfW_t = \bfPi_t \bfW_0+\bfR_t,
\]
where
${\bf\Pi}_t={\bfA}_t\cdots{\bfA}_1,\,\bfR_t=\sum_{i=1}^{t-1}\bfPi_{t,t-i+1}\bfB_{t-i}$
for $t\ge1$, and all vectors are column vectors. With this
interpretation we write 
\begin{align}
\label{eq:vSRE}
 (\bfW_1,\ldots,\bfW_h) = (\bfPi_1,\ldots,\bfPi_h)\bfW_0+(\bfR_1,\ldots,\bfR_h),
\end{align}
 where $(\bfPi_1,\ldots,\bfPi_h),\,(\bfR_1,\ldots,\bfR_h)$ have moment
 of order $\alpha_2$ with respect to the matrix norm and are independent
 of $\bfW_0$. Indeed, for all $t=1,\ldots,h$
 $\E|\bfPi_t|^{\alpha_2}<\infty$ and $\E|\bfR_t|^{\alpha_2}<\infty$. Due
 to Minkowski's and triangle inequalities, this implies that each random
 component in two matrices has $\alpha_2$ th moment. Thus $\alpha_2$ th
 moment with the matrix norm follows. 
Hence 
 \begin{align*}
  \lim_{x\to\infty} x^{\alpha_2}\P(|(\bfW_1,\ldots,\bfW_t)-(\bfPi_1,\ldots,\bfPi_t)\bfW_0|>x)=0
 \end{align*}
 holds 
 and an application of Breiman's lemma (\cite[Lemma C.3.1
 (1)]{buraczewski:damek:mikosch:2016}) concludes that
 $(\bfW_1,\ldots,\bfW_t)$ and $(\bfPi_1,\ldots,\bfPi_t)\bfW_0$ have the
 same tail behavior and are regularly varying with index
 $\alpha_2$. Finally, \eqref{eq:spectralpros} is concluded from the
 regular variations of 
 $(\bfPi_1,\ldots,\bfPi_t)\bfW_0$ and $\bfW_0$ together with the
 multivariate Breiman's
 lemma (\cite[Lemma C.3.1 (2)]{buraczewski:damek:mikosch:2016}) again. 
\end{proof}

\begin{remark}
 $($i$)$ In Lemma \ref{lem:differenttail}, although each component
 process has regular variation, we do not have a device to characterize
 the joint regular variation with different tail indices i.e.\ we
 could not characterize tail dependence between processes with different
 tail indices. Therefore
 we only provide that for coordinate-wise series. \\
 $($ii$)$ Regarding Proposition \ref{prop:fidm2} even
 $\alpha_2>\alpha_1>0$, it is possible to obtain the bivariate Basrak
 Segers limit representation \eqref{eq:spectralpros}, but then the spectral measure lies on one
 axis i.e.\ vectors $(W_{1,t},W_{2,t})$ has the same spectral behavior
 as $(W_{1,t},0)$. This is not desirable since in applications tail asymptotics of
 the both series are crucial. This is the reason why we adopt Lemma
 \ref{lem:differenttail}. 
\end{remark}

\subsection{Upper and lower bounds for constants}
Although we have expressions for constants $\ov c_{1}$ and $\wt c_{1}$ in Theorem \ref{thm:mainresult},
a direct numerical calculation of the quantities seems difficult. A further
research is needed for possible calculation of the constants as it is done in
\cite{mikosch:samorodnitsky:tafakori}. $\ov c_{1}$ of
\eqref{rep:constant2} is treated there but with i.i.d.\ 
sequence $(A_t,D_t)$. It remains to be
seen whether the method is applicable in our case or not. 
Alternatively, we derive the upper and lower bounds for $\wt c_{1}$. 

\begin{lemma}
 Assume the conditions of Theorem \ref{thm:mainresult} with $\alpha_1>\alpha_2$. If $\alpha_2>1$ then 
\[
 \E A_{2}^{\alpha_2} \le \wt c_{1} \le 
 (1-\tau^{1/\alpha_2})^{-\alpha_2}\, \E A_{2}^{\alpha_2}, 
\]
where $\tau=\E A_{1}^{\alpha_2}$ and if $\alpha_2 \le 1$ then 
\[
 (1-\tau^{1/\alpha_2})^{-\alpha_2}\, \E A_{2}^{\alpha_2} \le \wt
 c_{1} \le (1-\tau)^{-1}\,\E A_{2}^{\alpha_2}. 
\]
\end{lemma}   

\begin{proof}
 {\bf Case $\alpha_1>1$.}\ Since the limit and expectation are
 interchangeable, we work on 
\[
 \lim_{s\to \infty} \E\,\big(
 \sum_{i=1}^s \Pi_{0,2-i}^{(1)} A_{2,1-i} \Pi_{-i,1-s}^{(4)}
\big)^{\alpha_2}. 
\]
 We take the first term ($i=1$) i.e.\ $\Pi_{0,1}^{(1)}A_{2,0}\Pi_{-1,1-s}^{(4)}=A_{2,0}A_{4,-1}\cdots
 A_{4,1-s}$ in the sum, so that we
 obtain the lower bound $\E
 (A_{2,0}\Pi_{-1,1-s}^{(4)} )^{\alpha_2}=\E
 A_{2}^{\alpha_2}$. As for the upper bound, by the Minkowski's inequality
 together with independence of $\Pi_{0,2-k}^{(1)}$, $A_{2,1-k}$ and
 $\Pi_{-k,1-s}^{(4)}$, $1\le k\le s$ we obtain
\begin{align*}
 \E\,\big(
 \sum_{i=1}^s \Pi_{0,2-i}^{(1)} A_{2,1-i} \Pi_{-i,1-s}^{(4)}
\big)^{\alpha_2} & \le \Big(
 \sum_{i=1}^s \big(
 \E (\Pi_{0,2-i}^{(1)} A_{2,1-i} \Pi_{-i,1-s}^{(4)})^{\alpha_2}
 \big)^{1/\alpha_2}
 \Big)^{\alpha_2}\\
 &= \Big(
 (
 \E A_{2}^{\alpha_2})^{1/\alpha_2} \sum_{i=1}^s \tau^{(i-1)/\alpha_2}
 \Big)^{\alpha_2} \\
 &= \E A_{2}^{\alpha_2}\, \left(
 \frac{1-\tau^{s/\alpha_2}}{1-\tau^{1/\alpha_2}}
 \right)^{\alpha_2},
\end{align*}
where $\E A_4^{\alpha_1}=1,\,\E A_1^{\alpha_2}<1$ and $\E
 A_2^{\alpha_2}<\infty$. This concludes the first result by taking limit in $s$. \\
{\bf Case $\alpha_2 \le 1$.} We apply the triangle inequality for the
 upper bound and obtain 
\begin{align*}
 \E\,\big(
 \sum_{i=1}^s \Pi_{0,2-i}^{(1)} A_{2,1-i} \Pi_{-i,1-s}^{(4)}
\big)^{\alpha_2} \le 
 \sum_{i=1}^s \E \big(
 \Pi_{0,2-i}^{(1)}A_{2,1-i} \Pi_{-i,1-s}^{(4)}
 \big)^{\alpha_2} = \sum_{i=1}^s \tau^{i-1} \E A_{2}^{\alpha_2} = \frac{1-\tau^s}{1-\tau}\,\E A_{2}^{\alpha_2},
\end{align*}
which yields the result. The lower bound is implied by the reverse Minkowski's
 inequality. 
\end{proof}

\section{Application to bivariate \garch\ processes}
\label{aplication}
There are various extensions of a univariate GARCH
model to multivariate ones. We stick here to the 
{\em constant conditional correlation model} of Bollerslev
\cite{bollerslev:1990} and Jeanthequ \cite{jeantheau:1998}, which is
the most fundamental multivariate GARCH process. 
A bivariate series $\bfX_t=
(X_{1,t},X_{2,t})'$, $t\in \bbz$ has the
\garch\ structure if it satisfies:
\beam\label{eq:13}
 \bfX_t= \Sigma_t\, \bfZ_t\,,
\eeam
where $(\bfZ_t)$ constitutes an i.i.d.\ bivariate noise \seq\ and
\beao
\Sigma_t={\rm diag}(\sigma_{1,t},\sigma_{2,t})\,,
\eeao
with $\sigma_{i,t}$ being the (non-negative) volatility of $X_{i,t}$. 
We also assume that $\bfZ_t=(Z_{1,t},Z_{2,t})'$ has mean zero and its
covariance matrix (standard correlations) is
\begin{align*}
 P=\left(\barr{cc} 1 & \rho  \\
      \rho & 1 \earr\right),
\end{align*}
where $\rho=\mathrm{Corr}(Z_{1,t},Z_{2,t})$. 
The volatility process $\sigma_{i,t}$ is defined by the following
stochastic equation
\begin{align}
\begin{split}\label{eq:15}
\left(\barr{l}\sigma^2_{1,t}  \\  
\sigma^2_{2,t}\earr
\right)
=& \left(
\barr{l}\a_{01}  \\\a_{02}   \earr\right)
+\left(\barr{cc}\a_{11} & \a_{12}  \\
      \a_{21} & \a_{22}\earr \right)\, 
\left(\barr{l}X_{1,t-1}^2  \\X_{2,t-1}^2   \earr\right)
 + \left(\barr{cc}\b_{11} & \b_{12}  \\\b_{21} & \b_{22} \earr
 \right)\,\left(\barr{c}\sigma^2_{1,t-1}  \\\sigma^2_{2,t-1}\earr
  \right)\\
=& \left(
\barr{l}\a_{01}  \\\a_{02}   \earr\right)+\left(\barr{cc}\alpha_{11}Z_{1,t-1}^2+\beta_{11}&\alpha_{12}Z_{2,t-1}^2+
\beta_{12}\\
\alpha_{21}Z_{1,t-1}^2+\beta_{21}& \alpha_{22}Z_{2,t-1}^2+\beta_{22}
\earr\right)\,\left(\barr{l}\sigma_{1,t-1}^2\\\sigma_{2,t-1}^2\earr
\right),
\end{split}
\end{align}
where the second equality follows from \eqref{eq:13}.
Writing $\bfW_t=(\sigma^2_{1,t} \,,
\sigma^2_{2,t})'$, 
\begin{align}
\label{eq:20a}
\bfB_t=\left(\barr{c}\a_{01} \\ \a_{02} \earr\right) \quad \mathrm{and}\quad  
\bfA_t=\left(\barr{cc}\alpha_{11}Z_{1,t-1}^2+\beta_{11}&\alpha_{12}Z_{2,t-1}^2+
\beta_{12}\\
\alpha_{21}Z_{1,t-1}^2+\beta_{21}& \alpha_{22}Z_{2,t-1}^2+\beta_{22}
\earr\right):=\left(\barr{cc} A_{1,t} & A_{2,t} \\
A_{3,t} & A_{4,t}
\earr\right), 
\end{align}
we see that 
the process $(\bfW_t)$ is given by the SRE 
with vector-valued $\bfB_t$ and matrix-valued $\bfA_t$:
\beam
\label{eq:jan6b}
\bfW_t=\bfA_t\,\bfW_{t-1}+\bfB_t\,,\qquad t\in\bbz\,.
\eeam
There is a series of results about regular variation of GARCH processes. They are based on the Kesten theorem \cite{kesten:1973} and so, the tail indices of the component-wise series are always the same. We have in mind 
 Mikosch and \sta\ \cite{mikosch:starica:2000} for
 the univariate \garch\ model  and Basrak et
al~\cite{basrak:davis:mikosch:2002} 
for general univariate GARCH$(p,q)$ processes. 
The corresponding results for a vector \garch\ were obtained by  \sta\
\cite{starica:1999}. 
There is also a recent result by 
Fern\'andez and Muriel \cite{fernandez:muriel:2009}. Furthermore, tail dependencies for
bivariate \garch\ models could be captured by newly defined measure called 
(cross-) extremogram (Matsui and Mikosch
\cite{matsui:mikosch:2016}), which is proposed in \cite{davis:mikosch:2009} and further developed in
\cite{davis:mikosch:cribben:2012,davis:mikosch:zhao:2013}. 

However, in the financial models, we may be forced to go beyond Kesten's assumptions when some of the entries of $\bfA _t$ vanish. Then the results of the previous section become very handy and we are able to treat the \garch\ model with component-wise different extremes. Note that in Remark 3.2 of
\cite{matsui:mikosch:2016} 
the same assumption of upper triangle matrix was suggested for
component-wise different tail modeling, though they did not obtain the 
exact tail behavior. 

We assume that $\alpha_{21}=\beta_{21}=0$ in \eqref{eq:20a},
$\alpha_{0i}>0,\,i=1,2$ and $\alpha_{ij},\beta_{ij}>0$ for $(i,j) \neq (2,1)$. Then $\bfA _t$ becomes an upper triangular matrix and component-wise we have the following SREs  
\begin{align}
\label{compwisesres}
\begin{split}
 \sigma_{1,t}^2 &
 = A_{1,t}\sigma_{1,t-1}^2+D_t, \\ 
 \sigma_{2,t}^2 &= A_{4,t}\sigma_{2,t-1}^2+\alpha_{02},
\end{split}
\end{align}
where $D_t:=\alpha_{01}+A_{2,t}\sigma_{2,t-1}^2$.

Then the sufficient condition \eqref{epsilon_moment} for existence of the strictly stationary solution is:
there exists $\varepsilon >0$ such that 
\begin{align*}
 \max_{i}\,( \alpha_{ii}^\varepsilon \E
 Z_i^{2\varepsilon}+\beta_{ii}^\varepsilon ) <1\quad & \mathrm{if}\quad
 \varepsilon <1 \\
 \max_{i}\, ( \alpha_{ii}(\E
 Z_i^{2\varepsilon})^{1/\varepsilon}+\beta_{ii} )<1\quad & \mathrm{if}\quad
 \varepsilon \ge 1
\end{align*}
holds. Then by Theorem \ref{thm:mainresult} we have  

\begin{corollary}
\label{cor:fidm1}
 Consider the bivariate SRE \eqref{eq:jan6b} or equivalently SREs \eqref{compwisesres}. Assume that
 r.v.\ $\bfZ$ has Lebesgue density in $\R^2$ and
 there exist $\alpha_1,\alpha_2>0$ such that 
 \begin{align}
\label{cond:garchtail}
 \begin{split}
  \E A_{1}^{\alpha_1}=1\quad \mathrm{and} \quad \E A_{1}^{\alpha_1}\log^+
  A_{1} <\infty, \\
 \E A_{4}^{\alpha_2}=1\quad \mathrm{and} \quad \E A_{4}^{\alpha_2}\log^+
  A_{4} <\infty, 
 \end{split}
 \end{align}
 then
 \begin{align*}
 \P(\sigma_{1}^2 > x) \sim \Bigg \{
\begin{array}{ll}
\ov c_{1} x^{-\alpha_1} & \mathrm{if}\ \alpha_1
   < \alpha_2 \\
 \wt c_{1} x^{-\alpha_2} &\mathrm{if}\
   \alpha_1 > \alpha_2
\end{array}
\quad \mathrm{and}\quad \P(\sigma_{2}^2>x)\sim c_{2}x^{-\alpha_2}, 
 \end{align*}
 where the constants are given by \eqref{weq:tails} - \eqref{rep:constant1}.
\end{corollary}

\begin{proof}
 We have to check the conditions of Theorem \ref{thm:mainresult}. Since
 each element of $\bfB_t$ is a positive constant, this together with
 condition \eqref{cond:garchtail} implies the condition
 \eqref{comptailcondi}. In view of \eqref{eq:20a}, $A_2$ has the
 $\alpha_2$th moment since the random component is the same as that in $A_4$.
 Then, since $\E A_{2}^{\alpha_1}<(\E
 A_{2}^{\alpha_2})^{\alpha_1/\alpha_2}<\infty$ for $\alpha_1<\alpha_2$
 and $\E A_{2}^{\alpha_2}<\infty$ for $\alpha_2<\alpha_1$ we have $\E
 A_{2}^{\min (\alpha_1,\alpha_2)}<\infty$. Obviously $\bfA$
 has Lebesgue density. Therefore all conditions are satisfied. 
\end{proof}

Now we are going to characterize regular variation of stationary \garch\
 process. We do not apply Lemma \ref{lem:differenttail} and Proposition \ref{prop:fidm2} directly because the corresponding  SRE is satisfied by the volatility vector not by the GARCH process itself. Therefore, some additional work is needed.
 First we assume that $\alpha_1<\alpha_2$ and for $h\ge 0$ and $i=1,2$
 we define the following lagged vectors. 
  \begin{align*}
  \bfX_{i,h} &=(X_{i,1},\ldots,X_{i,h}),\\
  \bfX_{i,h}^{(k)} &= (|X_{i,1}|^k,\ldots,|X_{i,h}|^k),\,k=1,2,\\
  \bfY_{1,h} &= (|Z_{1,1}| (\Pi_{1}^{(1)})^{1/2},\ldots,|Z_{1,h}|
  (\Pi_{h}^{(1)})^{1/2}), \\
   \bfY_{2,h} &= (|Z_{2,1}| (\Pi_{1}^{(4)})^{1/2},\ldots,|Z_{2,h}|
  (\Pi_{h}^{(4)})^{1/2}). 
 \end{align*}
 In what follows for a matrix or a vector $\bfA$ and a constant $\alpha>0$, 
 $\bfA^\alpha$ denotes component-wise $\alpha$th power of $\bfA$. 

\begin{proposition}
 \label{prop:spectral1}
Suppose that the conditions of Corollary 
 \ref{cor:fidm1} with $\alpha_1<\alpha_2$ are satisfied and that additionally $\bfZ$ is symmetric. Then random vector $\bfX_{i,h}$ is regularly varying with
 index $2\alpha_i,\,i=1,2$. The spectral measure is given by the
 distribution of the vector
 \begin{align}
 \label{eq:spectral1}
  (r_{i,1}\theta_{i,1},\ldots,r_{i,h}\theta_{i,h}), 
 \end{align}
 where r.v. $\bfTh_{i,h}=(\theta_{i,1},\ldots,\theta_{i,h})\in
 \bbs^{h-1}$ has distribution 
 \[
  \P(\bfTh_{i,h}\in\cdot)= \frac{\E|\bfY_{i,h}|^{2\alpha_i}\,
 {\bf1}\,(\bfY_{i,h}/|\bfY_{i,h}|\in \cdot)}{\E|\bfY_{i,h}|^{2\alpha_i}},
 \]
 and $(r_{i,t})$ is a sequence of the Bernoulli r.v.'s 
 independent of $\bfTh_{i,h}$ such
 that $\P(r_{i,t}=\pm 1)=0.5$. 
\end{proposition} 

\begin{proof}
  Since proofs for the series $(X_{1,t})$ and $(X_{2,t})$ are almost the
  same, it suffices to see that for $(X_{1,t})$.
 In view of \eqref{eq:wpr}, we may write 
 \begin{align*}
  \bfX_{1,h}^{(2)} &=
  (Z_{1,1}^{2}\sigma_{1,1}^2,Z_{1,2}^2\sigma_{1,2}^2,\ldots,Z_{1,h}^2
  \sigma_{1,h}^2) \\
 &= (Z_{1,1}^2 \Pi_{1}^{(1)},\ldots,Z_{1,h}^2
  \Pi_{h}^{(1)})\,\sigma_{1,0}^2
  +(Z_{1,1}^2R_{1,1},\ldots,Z_{1,h}^2R_{1,h}). 
 \end{align*}
 We recall that $Z_{1,t}^2$ and $\Pi_{t}^{(1)}$ are independent and have moment
 of order $\alpha_2$ so that $|Z_{1,t}|(\Pi_{t}^{(1)})^{1/2}$ has the $2
 \alpha_1$th moment. Moreover,  $D_k=\alpha_{01}+A_{2,k}\sigma_{2,k-1}^2$  in 
\[
 R_{1,t}=\sum_{i=1}^{t-1} A_{1,t}\cdots A_{1,t+1-i}D_{t-i}+D_t 
\]
 has moment of order $\alpha_2>\alpha_1$ and $Z_{1,t}^2$ and
 $R_{1,t}$ are independent. Therefore, Minkowski's or triangle inequality
 implies that $|Z_{1,t}|(R_{1,t})^{1/2}$ has moment of $2\alpha_1$. 
 Hence 
 \[
  \lim_{x\to\infty} x^{2\alpha_1}
 \P(|\bfX_{1,h}^{(1)}-\bfY_{1,h}\sigma_{1,0}|>x)=0 
 \]
 holds and an applying the multivariate Breiman's lemma we conclude
 that $\bfX_{1,h}^{(1)}$ and $\bfY_{1,h}\sigma_{1,0}$ have the same tail
 behavior and are regularly varying with index $2\alpha_1$. 
Now similarly as in the proof of
 Lemma \ref{lem:differenttail}, as $y\to\infty$
 \begin{align*}
  \frac{\P(|\bfX_{1,h}^{(1)}|>xy,\,\bfX_{1,h}^{(1)}/|\bfX_{1,h}^{(1)}|
  \in\cdot)}{\P(|\bfX_{1,h}^{(1)}|>y)} &\sim
  \frac{\P(\sigma_{1,0}|\bfY_{1,h}|>xy,\,
  \bfY_{1,h}/|\bfY_{1,h}|
  \in\cdot)}{\P(\sigma_{1,0} |\bfY_{1,h}|>y)} \\
 &\sim x^{-2 \alpha_1}
  \frac{\E|\bfY_{1,h}|^{2\alpha_1}\,{\bf1}\,(\bfY_{1,h}/|\bfY_{1,h}|
  \in\cdot )
  }{\E |\bfY_{1,h}|^{2\alpha_1}} \\
 & \sim x^{-2 \alpha_1}\P(\bfTh_{1,h}\in \cdot). 
 \end{align*}
 Write 
 \[
  \bfX_{1,h}=(\sign(Z_{1,1})|X_{1,1}|,\ldots,\sign(Z_{1,h})|X_{1,h}|),  
 \]
 then by symmetry of $\bfZ$, the sequence $(\sign(Z_{1,t}))$ is
 independent of $(|X_{1,t}|)$. Hence by Proposition 5.13 of
 \cite{davis:mikosch:basrak:1999}, $\bfX_{1,h}$ is regularly varying with index $2\alpha_1$ and the
 spectral measure is given by that of \eqref{eq:spectral1}. 
\end{proof}

 In the case $\alpha_1>\alpha_2$ we are interested in the spectral process
 as done in
 Proposition \ref{prop:fidm2}. However, we dare to use the spectral
 process by $\bfW_t=(W_{1,t},W_{2,t})'$ since this gives explicit
 expression and since the original definition
 may yield only closed form representation as in
 \cite{matsui:mikosch:2016}. 
 
\begin{proposition}\label{spectralgarch}
Assume that $\alpha_1>\alpha_2$ and the conditions of Corollary \ref{cor:fidm1} are satisfied. Let $h\ge 0$, then $(\bfX_t)$ is regularly varying with index $2\alpha_2$ in the sense of \eqref{eq:jan6a}. In particular, with
 $\bfW_t=(\sigma_{1,t}^2,\sigma_{2,t}^2)'$, we have 
 \begin{align}
 \label{eq:fidmg2}
  \P(x^{-1/2}(\bfX_1,\ldots,\bfX_h)\in \cdot \mid |\bfW_0|>x)\stackrel{w}{\to} 
 \P\big(V({\rm diag}(\bfPi_1\bfTh_0))^{1/2}\bfZ_1,\ldots,({\rm
  diag}(\bfPi_h\bfTh_0))^{1/2}\bfZ_h)\in \cdot \big),
 \end{align}
 where $\P(V>x)=x^{-2\alpha_2}$ for $x>1$ and $V$ is independent of
 $(\bfTh_0,(\bfPi_1,\ldots,\bfPi_h))$ and $(\bfZ_1,\ldots,\bfZ_h)$. 
\end{proposition}

\begin{proof}
 Write $\Sigma_t=({\rm diag}(\bfW_t))^{1/2}$ so that
 $\bfX_t=\Sigma_t\bfZ_t$. First we approximate $\bfX_t$ by
 $\wt \Sigma_t\bfZ_t$ where
 $\wt \Sigma_t=({\rm diag}(\bfPi_t\bfW_0))^{1/2}$. In view of
 \eqref{eq:vSRE} the triangle inequality yields 
 \[
  |(\Sigma_t-\wt\Sigma_t)\bfZ_t| \le |\bfR_t|^{1/2} |\bfZ_t|. 
 \]
 Since $|\bfR_t|^{1/2} |\bfZ_t|$ has $2\alpha_2$th moment, we have 
 \[
  \lim_{x\to\infty}x^{2\alpha_2}\P \big(|(\Sigma_1\bfZ_1,\ldots,\Sigma_h\bfZ_h)-(\wt
 \Sigma_1\bfZ_1,\ldots,\wt \Sigma_h\bfZ_h)|>x\big)=0,
 \]
 which together with the Breiman's lemma implies that
 $(\bfX_1,\ldots,\bfX_h)$ and $(\wt \Sigma_1\bfZ_1,\ldots,\wt \Sigma_h
 \bfZ_h)$ have the same tail behavior and are regularly varying with
 index $2\alpha_2$. For \eqref{eq:fidmg2} we recall from Proposition
 \ref{prop:fidm2} that 
 $\P(x^{-1}\bfW_0\in \cdot \mid |\bfW_0|>x)\stackrel{w}{\to}\P(V^2
 \bfTh_0\in\cdot)$ with $\P(V^2>x)=x^{-\alpha_2},\,x>1$ and $V^2$
 and $\bfTh_0$ are independent. Due to the multivariate Breiman's lemma,  
 \[
  \P(x^{-1}(\bfPi_1\bfW_0,\ldots,\bfPi_h\bfW_0)\in \cdot \mid
 |\bfW_0|>x)\stackrel{w}{\to} P(V^2 (\bfPi_1,\ldots,\bfPi_h)\bfTh_0\in
 \cdot). 
 \]
 Finally, applying the continuous mapping
 theorem and another Breiman's lemma, we obtain the
 convergence of 
 \[
  \P\big(
 x^{-1/2}( ({\rm diag}(\bfPi_1\bfW_0))^{1/2}\bfZ_1,\ldots,({\rm
 diag}(\bfPi_h\bfW_0))^{1/2}\bfZ_h)\in \cdot \mid |\bfW_0|>x
 \big)
 \]
 to the right hand side of \eqref{eq:fidmg2}. 
\end{proof} 

\noindent {\bf Acknowledgment:} 
This work started when E. Damek and M. Matsui met during the 
workshop 
``Mathematical Foundations of Heavy Tailed Analysis'' 2015 
in Copenhagen. Later on it was continued when M. Matsui was 
visiting Institute of Mathematics of Wroclaw University.  
The authors would like to thank both institutions for their hospitality
and financial support.

\end{document}